\documentclass[11pt]{article}
\RequirePackage{fix-cm}
\usepackage{graphicx}
\usepackage{amsmath}
\usepackage{amsthm}
\usepackage{amsfonts}
\usepackage{amssymb}
\usepackage{graphicx}
\usepackage{color}
\usepackage{rotating}
\usepackage{authblk}
\usepackage{url}
\usepackage{tcolorbox}
\usepackage{tikz,pgfplots} 
\usetikzlibrary{arrows,shadings,shadows,automata,shapes,calc,positioning,patterns,decorations.markings,decorations.pathmorphing,snakes,plotmarks}

\usepackage{soul}
\usepackage{footnote}
\usepackage{geometry}

\newtheorem{observation}{Observation}

\newcommand{\Z}{\mathbb{Z}}
\newcommand{\R}{\mathbb{R}}

\newcommand{\e}{\varepsilon}

\DeclareMathOperator{\conv}{conv}
\DeclareMathOperator{\cone}{cone}
\DeclareMathOperator{\clcone}{clcone}
\DeclareMathOperator{\rank}{rank}

\newcommand{\A}{\mathcal{A}}
\newcommand{\C}{\mathcal{C}}

\newcommand{\Ai}{A^i}
\newcommand{\kvi}{$k$-vi}

\newcommand{\bigL}{L}
\newcommand{\LPstar}{LP^*}

\newtheorem{theorem}{Theorem}
\newtheorem{proposition}{Proposition}

\usepackage{authblk}

\usepackage{lipsum}

\makeatletter
\renewcommand\footnoterule{%
  \kern-3\p@
  \hrule\@width \textwidth
  \kern2.6\p@}
\makeatother

\makeatletter
\renewcommand*{\@fnsymbol}[1]{\ensuremath{\ifcase#1\or *\or \dagger\or \ddagger\or **\or \mathsection\or \mathparagraph\or \|\or  \dagger\dagger
   \or \ddagger\ddagger \else\@ctrerr\fi}}
\makeatother

\begin{document}

\title{Aggregation-based cutting-planes for packing and covering integer programs
}


\author[1]{Merve Bodur\thanks{merve.bodur@gatech.edu}}
\author[2]{Alberto Del Pia\thanks{delpia@wisc.edu}}
\author[1]{Santanu S. Dey\thanks{santanu.dey@isye.gatech.edu}}
\author[3]{Marco Molinaro\thanks{mmolinaro@inf.puc-rio.br}} 
\author[1]{Sebastian Pokutta\thanks{sebastian.pokutta@isye.gatech.edu}}
\affil[1]{\small School of Industrial and Systems Engineering, Georgia Institute of Technology} 
\affil[2]{\small Department of Industrial and Systems Engineering \& Wisconsin Institute for Discovery, University of Wisconsin-Madison}
\affil[3]{\small Computer Science Department, Pontifical Catholic University of Rio de Janeiro}

\maketitle

\begin{abstract}
In this paper, we study the strength of Chv\'atal-Gomory (CG) cuts and more generally \emph{aggregation cuts} for packing and covering integer programs (IPs). Aggregation cuts are obtained as follows: Given an IP formulation, we first generate a single implied inequality using aggregation of the original constraints, then obtain the integer hull of the set defined by this single inequality with variable bounds, and finally use the inequalities describing the integer hull as cutting-planes. Our first main result is to show that for packing and covering IPs, the CG and aggregation closures can be \emph{2-approximated} by simply generating the respective closures for each of the original formulation constraints, without using any aggregations.
On the other hand, we use computational experiments to show that aggregation cuts can be arbitrarily stronger than cuts from individual constraints for general IPs. The proof of the above stated results for the case of covering IPs with bounds require the development of some new structural results, which may be of independent interest. Finally, we examine the strength of cuts based on $k$ different aggregation inequalities simultaneously, the so-called multi-row cuts, and show that 
\emph{every} packing or covering IP with a large \emph{integrality gap} also has a large \emph{$k$-aggregation closure rank}. In particular, this rank is always at least of the order of the logarithm of the integrality gap. 

\end{abstract}

\section{Introduction}\label{sec:intro}
Cutting-planes are central to state-of-the-art integer programming (IP) solvers \cite{bixby2004,Lodi2009}. While different methods have been developed to generate various families of cutting-planes~\cite{marchand:ma:we:wo:2002,RichardDey}, several of the most important families are obtained through the aggregation of the original constraints of the problem. These are special types of what we call \emph{aggregation cuts}, which are those generated as follows: given an IP formulation, we first obtain a single implied inequality by aggregating the original constraints, and then generate a cut valid for the integer hull of the set defined by this single inequality together variable bounds.

It is easy to see that \emph{Chv\'atal-Gomory (CG) cuts} are aggregation cuts: in fact, each CG cut is precisely the integer hull of the set defined by one aggregated inequality \emph{without} variable bounds. Aggregation cuts include many other classes of cuts, such as lifted knapsack covers inequalities~\cite{wolsey:1975,zemel:1978} and weight inequalities \cite{weismantel19970}. The set of all aggregation cuts have been studied empirically~\cite{FukasawaG11}, but to the best of our knowledge no theoretical study is present.

	Given the ubiquity of aggregation cuts, it is important to better understand the role of aggregation in integer programming. Of direct practical importance is to understand which aggregations are most useful. Another interesting direction, which we pursue here, is to understand in which cases aggregation is most helpful and what are the limitations of using aggregation-based cuts.
	
	In this paper, we examine the strength of aggregation cuts for \emph{packing} and \emph{covering} IPs. Our main result is that for these classes of problems, even considering all infinitely many aggregations offers limited help. More precisely, we show that the CG and more generally aggregation closures can be 2-approximated by simply generating the respective closures for each of the original constraints, without using any aggregations. Therefore, for these problems, in order to obtain cuts that are much stronger than original constraint cuts, one needs to consider more complicated cuts that cannot be generated through aggregations; see for example the results in \cite{deyMolinaroWang:2016}.
	
		 We also examine the strength of cuts based on $k$ different aggregated inequalities simultaneously (also called \emph{multi-row cuts}) for packing and covering problems. We show that \emph{every} packing or covering IP with a large integrality gap also has a large \emph{$k$-aggregation closure rank}; more precisely, for a fixed $k$, this rank is always at least of the order of the logarithm of the integrality gap. This again points to the relative weakness of aggregation cuts for packing or covering problems.
	
	Finally, simple examples show that these results are not true for general IPs, where aggregations can produce significant benefits. We provide further empirical evidence for this fact based on randomly generated general IPs and \emph{market split} instances \cite{cornuejols:da:1999}. 
	From cut selection perspective, the insight here is that for packing and covering problems, using aggregation cuts may provide limited benefit over using cuts generated from only the original constraints, while aggregation cuts may produce significant value for general IPs.
		 
\paragraph{Organization.} In Section \ref{sec:DefnStatements} we provide definitions and statements of all our main results and discuss them in more detail; we also present results from the computational experiments. In Section \ref{sec:Discussion} we state some open questions. Finally, in Section \ref{sec:Packing} and Section \ref{sec:Covering} we present the proofs for results concerning the packing and covering cases, respectively.

\section{Definitions and statement of results}
\label{sec:DefnStatements}

	\subsection{Definitions}

For an integer $n$, we use the notation $[n]$ to describe the set $\{1, \dots, n\}$. 
For $i \in [n]$, we denote by $e_i$ the $i^{\text{th}}$ vector of the standard basis of $\mathbb{R}^n$.
The convex hull of a set $S$ is denoted as $\textup{conv}(S)$, its conic hull is denoted as $\textup{cone}(S)$, and its closed conic hull is donated as $\clcone(S)$. For a set $S \subseteq \mathbb{R}^n$ and a positive scalar $\alpha$ we define $\alpha S:= \{\alpha u\, |\, u \in S\}$. 

	\paragraph{Packing and covering.} A \emph{packing polyhedron} is of the form $\{x \in \R^n_+ \mid Ax \le b \}$ where all the data $(A,b)$ is non-negative and rational. While polyhedral sets are the main object of study here, we will also need non-polyhedral ones.\footnote{This is needed because we do not know whether the aggregation closure is polyhedral.} So a \emph{packing set} is one of the form $\{x \in \R^n_+ \mid \Ai x \le b_i \ \forall i \in I\}$ where each $(\Ai, b_i) \in (\R_+^{1 \times n}, \R_+)$ and $I$ is an arbitrary set. 
%

	Similarly, a \emph{covering polyhedron with bounds} is of the form $\{x \in \R^n_+ \mid Ax \ge b, ~x \le u \}$ where all the data $(A,b,u)$ is non-negative and rational. We assume a component of $u$ is either finite and integral, or infinite. If all upper bounds take the value of infinity, then we simply call the set a \emph{covering polyhedron}. In the non-polyhedral case, a \emph{covering set with bounds} has the form $\{x \in \R^n_+ \mid A^i x \ge b_i \ \forall i \in I, ~x \le u \}$ with $(\Ai, b_i) \in (\R_+^{1 \times n}, \R_+)$ and $u$ satisfying the same assumptions as above, but $I$ is an arbitrary set.

	\paragraph{Closures.} Given a polyhedron $Q$, we are interested in cuts for the pure integer set $Q \cap \Z^n$.
	We use $\C(Q)$ and $Q^I$ to denote the CG closure and the convex hull of integer feasible solutions of $Q$, respectively (see, e.g., \cite{ConCorZam14b} for definitions). Moreover,  given a packing polyhedron $Q = \{x \in \R^n_+ \mid Ax \le b\}$, we define its \emph{aggregation closure} as $$\A(Q) := \bigcap_{\lambda \in \R^m_+} \conv(\{x \in \Z^n_+ \mid \lambda^\top A x \le \lambda^\top b\}).$$ Similarly, for a covering polyhedron $Q = \{x \in \R^n_+ \mid Ax \ge b, \ x \le u\}$ its aggregation closure is defined as $$\A(Q) := \bigcap_{\lambda \in \R^m_+} \conv(\{x \in \Z^n_+ \mid \lambda^\top A x \ge \lambda^\top b, \ x \le u\}).$$ Notice that we leave the bounds of the variables disaggregated, which gives a stronger closure than if we had just kept the non-negativity inequalities disaggregated. It is clear that all CG cuts are aggregation-based cuts, namely $\C(Q) \supseteq \A(Q)$.
	
	In order to understand the power of aggregations for generating cuts of these families, we define the 1-row (or non-aggregated) version of these closures. The \emph{1-row CG closure} $1\C(Q)$ is defined as the intersection of the CG closures of the individual inequalities defining $Q$, together with variable bounds; more precisely, for a packing polyhedron $Q$ $$1\C(Q) = \bigcap_{i \in [m]} \C(\{x \in \R^n_+ \mid A^ix \le b_i\}),$$ and for a covering polyhedron with bounds we have $$1\C(Q) = \bigcap_{i \in [m]} \C(\{x \in \R^n_+ \mid A^ix \ge b_i, \ x \le u\}),$$ where $\Ai$ denotes the $i^{\text{th}}$ row of $A$. The \emph{1-row closure} $1\A(Q)$ is defined analogously, simply replacing the operator $\C(.)$ by $\A(.)$.
	
Given a packing polytope $Q$ and a non-negative objective function $c \in \R^n_+$, we define $$z^{1\C} := \max\{c^\top x \mid x \in 1\C(Q)\}$$ as the optimal value over the closure $1\C(Q)$, and similarly for all the other closures, namely $z^{1\A},z^{\A},z^{\C}$. Moreover, we use $z^I$ and $z^{LP}$ to denote the optimal objective function value over $Q$ and its linear programming (LP) relaxation, respectively. For covering integer sets (with bounds) ``$\max$'' is replaced with ``$\min$''.

	We can generalize the aggregation closure to consider simultaneously $k$ aggregations, where $k \in \Z$ and $k \geq 1$. More precisely, for a covering polyhedron $Q$ the \emph{$k$-aggregation} closure is defined as 
	\begin{align*}
		\A_k(Q) := \bigcap_{\lambda^1, \ldots, \lambda^k \in \R^m_+} 
		\conv(\{ x \in \Z^n_+ \mid 
		(\lambda^j)^\top A x \ge (\lambda^j)^\top b~
		\forall j\in[k],  
		~x \le u \}),
	\end{align*}
	and the definition is similar for the packing case.
	
	More generally, given a packing \emph{set} $Q$, its $k$-aggregation closure $\A_k(Q)$ is defined as the intersection of all sets $\conv(\{x \in \Z^n_+ \mid D^j x \le f_j ~\forall j \in [k]\})$ where each of the $k$ rows $D^j x\le f_j$ is a valid inequality for $Q$ with non-negative coefficients. Similarly, given a covering set with bounds $Q$, $\A_k(Q)$ is defined as the intersection of all sets $\conv(\{x \in \Z^n_+ \mid D^j x \ge f_j ~\forall j \in [k], ~x \le u\})$ where each $D^j x\ge f_j$ is a valid inequality for $Q$ with non-negative coefficients. Notice that these definitions are independent of the representation of $Q$, and in the polyhedral case a duality argument shows that they are equivalent to the aggregation-based ones given above. 

	The \emph{$k$-aggregation closure rank}, denoted by $\rank_{\A_k}(Q)$, is defined in the standard way: it is the minimum number of applications $\A_k(\A_k(\ldots \A_k(Q) \ldots))$ of $A_k$ in order to obtain the convex hull of $Q$. Notice that if $Q$ is a packing (resp. covering) set, $\A_k(Q)$ is a packing (resp. covering) set, so iterating the closure $\A_k$ is a well-defined operation; we will formally verify this later. Moreover, since the CG rank is always finite \cite{schrijver1980cutting}, and the aggregation closure of $Q$ is contained in the CG closure of $Q$, we have that $\rank_{\A_k}(Q)$ is always finite.

	\paragraph{Approximation.} Given two packing sets $U \supseteq V$, we say that $U$ is an \emph{$\alpha$-approximation} of $V$
if for all non-negative objective functions $c \in \R^n_+$ we have $$\max\{c^\top x \mid x \in U\} \le \alpha \cdot \max\{c^\top x \mid x \in V\}.$$ 
Notice that since $U \supseteq V$, we have $\alpha \ge 1$.
%
%
Similarly, for a covering polyhedron (with bounds) $Q$, 
given two covering sets $U \supseteq V$
we say that $U$ is an $\alpha$-approximation of $V$ if for all $c \in \R^n_+$ we have $$\min\{c^\top x \mid x \in U\} \ge \frac{1}{\alpha} \cdot \min\{c^\top x \mid x \in V\}.$$
	\subsection{Statement of results}
	\subsubsection{Packing} 

	The following is our main result comparing closures with their 1-row counterparts. 
	
	\begin{theorem}\label{thm:pack}
	Consider a packing polyhedron $Q$.
	Let $\mathcal{M}$ be any of the closures $\A$ (aggregation) or $\C$ (CG). Then  $1\mathcal{M}(Q)$ is a 2-approximation of $\mathcal{M}(Q)$.
		
	Moreover, this bound is tight, namely for every $\e > 0$ there is a packing polyhedron $Q$ such that $1\mathcal{M}(Q)$ is not a $(2-\e)$-approximation of $\mathcal{M}(Q)$.
\end{theorem}
	
In the proof of Theorem~\ref{thm:pack} we introduce a special polyhedral relaxation of the convex hull of a packing polyhedron $Q$ that we call the \emph{pre-processed LP}. In this pre-processed LP, we examine if $A_{ij} > b_i$ for some $i \in [m],~j\in [n]$, in which case we set $x_j$ to 0. The optimal objective function value of the pre-processed LP is denoted by $z^{\LPstar}$.
Two key arguments of our proof involve this polyhedral relaxation:
(i) in Proposition~\ref{prop:pack_1CandLP} we prove that both 1-row CG closure and 1-row closure of $Q$ are contained in the pre-processed LP;
(ii) in Proposition~\ref{prop:pack2}, we show that the pre-processed LP is a $2$-approximation to $\A(Q)$;  
see Figure \ref{fig:packProofSch}.

	\begin{figure}[h]
	\centering
	\includegraphics{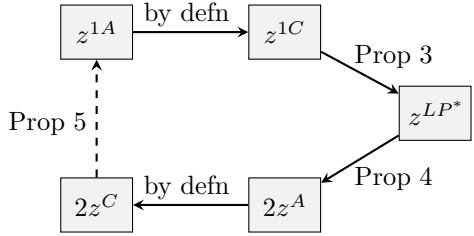}
	\caption{Relations used in the proof of Theorem \ref{thm:pack}. A straight arrow from $I$ to $J$ denotes the relation $I \leq J$, while a dashed arrow shows the existence of a tight example. Proposition numbers proving the relations are given on the arrows for the ones that are not implied by definitions (``by defn").}
	\label{fig:packProofSch}
	\end{figure}

The key take away of Theorem~\ref{thm:pack} is that for packing problems one can approximate the CG and aggregation closure by just considering their 1-row counterpart. We next show that this is not true in general. 

\begin{theorem}\label{thm:noncoverpack}
	Let $\mathcal{M}$ be any of the closures $\A$ (aggregation) or $\C$ (CG). Then there is a family of (non-packing/non-covering) polyhedra for which $1\mathcal{M}$ is an arbitrarily bad approximation to $\mathcal{M}$, namely for each $\alpha \ge 0$ there is a polyhedron $P$ such that $1\mathcal{M}(P)$ is not an $\alpha$-approximation of $\mathcal{M}(P)$.
\end{theorem}

The proof of Theorem \ref{thm:noncoverpack} gives a family of polyhedra in $\R^2$ where $\frac{z^{LP}}{z^I}$ can be arbitrarily large, but the CG rank is one.

On the other hand, we relate the integrality gap to the aggregation-closure rank. While there are many lower bounds on CG ranks (and reverse CG rank)~\cite{ChvatalCH89,ConfortiPSFGCG15,PokuttaS11,RothvossS13}, to the best of our knowledge there are no results for the aggregation closure. Moreover, our next lower bound adds to the list of few results~\cite{PokuttaS11,SinghT10} that relate integrality gaps to rank. 

\begin{theorem}\label{thm:packrankgen}
Let $Q = \{ x \in \R_+^n \mid A x \leq b \}$ be a packing polyhedron with $A_{ij} \leq b_i$ for all $i \in [m],~ j\in [n]$. Then, $\rank_{\A_k}(Q) \geq \left\lceil\frac{\textup{log}_2\left( \frac{z^{LP}}{z^I}\right)}{\textup{log}_2(k + 1)}\right\rceil$ for $k \geq 1$. Moreover, this bound is tight for $k=1$, that is, there is a packing polyhedron $Q$ with $\rank_{A_1}(Q) \le O\left(\textup{log}_2\left( \frac{z^{LP}}{z^I}\right)\right)$.
\end{theorem}

Theorem \ref{thm:packrankgen} shows that as long as we use information from a fixed number of constraints, packing IPs can take many rounds of cuts to obtain the integer hull. We remark that this result actually holds for packing sets defined by \emph{infinitely many} inequalities, see the proof of Theorem \ref{thm:packrankgen}. We also note that it can be verified that $\A_k$ is an \emph{admissable} cutting-plane operator, and therefore there exist 0-1 polytopes (with empty integer hulls) with rank $\Omega(\frac{n}{\log n})$ \cite{pokutta2010rank}.


\subsubsection{Covering}

	We show that the 1-row closures also provide a good approximation to the full closures in the case of covering polyhedra (with bounds). 
	
	\begin{theorem}\label{thm:cover}
	Consider a covering polyhedron (with bounds) $Q$. Let $\mathcal{M}$ be any of the closures $\A$ (aggregation) or $\C$ (CG). Then $1\mathcal{M}(Q)$ is a 2-approximation of $\mathcal{M}(Q)$.
	
	Moreover, this bound is tight, namely for every $\e > 0$ there is a covering polyhedron (with  bounds) such that $1\mathcal{M}(Q)$ is not a $(2-\e)$-approximation of $\mathcal{M}(Q)$.
\end{theorem}

 The key arguments of our proof are presented in Figure \ref{fig:coverProofSch}. As in the packing case, the main handle to prove this result is a pre-processed version of the LP. For covering polyhedra (i.e., without bounds), this pre-processing is natural: If $A_{ij} > b_i$ for some $i \in [m],~ j \in [n]$, since we are interested only in integer solutions, it is sufficient to replace $A_{ij}$ by $b_i$ to obtain a tighter LP. For covering polyhedra with bounds, the pre-processing LP is heavier and is given by adding all the \emph{knapsack-cover (KC) inequalities} \cite{carr2000strengthening,wolsey:1975}. We note that in the absence of bounds, this LP with the KC inequalities reduces to the pre-processed LP discussed above. The optimal objective function value of the LP with the KC inequalities is denoted as $z^{KC}$.
 
 	\begin{figure}[h]
	\centering
	\includegraphics{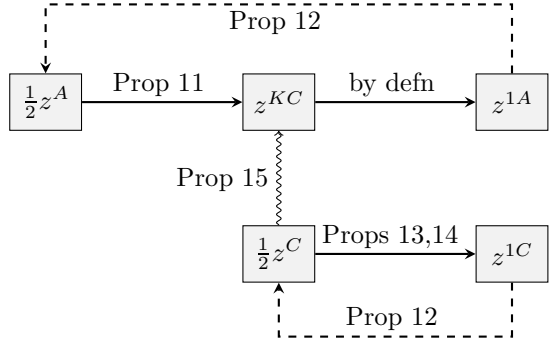}
	\caption{Relations used in the proof of Theorem \ref{thm:cover}. A straight arrow from $I$ to $J$ denotes the relation $I \leq J$, a dashed arrow shows the existence of a tight example, and a snake arrow means that the ratio could be arbitrarily large. Proposition numbers proving the relations are given on the arrows for the ones that are not implied by definitions (``by defn").}
	\label{fig:coverProofSch}
	\end{figure}

Unlike in the packing case, the statement of Theorem \ref{thm:cover} regarding the CG closure actually requires a different and much more involved proof. In fact, in this case we show that the LP with the KC inequalities cannot be used to prove this result: there are instances where the CG closure is arbitrarily weaker than the LP with the KC inequalities (see the snake arrow in Figure \ref{fig:coverProofSch}), i.e., for any $\bigL >0$ there exists an instance where $\bigL z^{\C} \leq z^{KC}$. Therefore $z^{\C}$ does not approximate $z^{KC}$ well, and hence it does not  approximate $z^{\A}$ well. We also refer the reader to \cite{bienstock2006approximate} for other techniques on approximating fixed rank CG closures for 0-1 covering IPs.

We note that in the proof of Theorem \ref{thm:cover}, we require some preliminary structural results regarding covering polyhedra with bounds, which may be of independent interest. See Propositions \ref{prop:coverUpward}-\ref{prop:commCover} in Section \ref{subsec:CoveringProperties}. 

	
	As in the packing case, we can also prove that a large integrality gap implies large rank for the $k$-aggregation closure. Interestingly, the denominator of the lower bound scales as $\log \log k$; this is because the largest integrality gap in a covering problem with $m$ constraints is $O(\log m)$ (see \cite{vazirani2013approximation}).
	
\begin{theorem} \label{thm:rankCover}
Consider a covering polyhedron $Q = \{ x \in \R_+^n \mid Ax \geq b\}$, where $A$ and $b$ satisfy $A_{ij} \leq b_i$ for all $i \in [m],~j \in [n]$. Then, the rank of the k-aggregation closure of $Q$ is at least $\left \lceil  \left(\frac{\textup{log}_2\left(\frac{z^I}{z^{LP}}\right)}{3+\textup{log}_2\textup{log}_2(2k)}\right)\right \rceil$.
\end{theorem} 

As in the packing case, the proof of Theorem \ref{thm:rankCover} shows that this result also holds for covering sets defined by \emph{infinitely many} inequalities.
\subsubsection{Computational experiments}

Theorem \ref{thm:noncoverpack} shows that for general IPs (not packing or covering problems), the 1-row version of the closures may not provide an approximation to the full closure, thus indicating the usefulness of aggregation-based cuts. In order to understand this phenomenon, we conduct an empirical study using CG cuts. Experimenting with CG cuts is convenient due to the availability of reasonably robust CG cut separating algorithm \cite{fischetti:lo:2007}. We use IBM ILOG Cplex 12.6 as the LP/MILP solver. We study two classes of instances: random instances and the so-called market split instances.

\paragraph{Random instances.} We generate instances of the following form:
$$\max \Big\{ \sum_{j \in [n]} x_j \mid Ax = b,~0 \leq x \leq u \Big\},$$
where
\begin{enumerate}
\item We consider instances with $n \in \{ 10,12,14,16 \}$ variables and $m = \lfloor n/2 \rfloor$ equality constraints.
\item We choose $M = 50$ and set $u_j = M / 2$ for all $j \in [n]$.
\item For any $i \in [m], j \in [n]$, we let $A_{ij}=0$ with probability 0.5. Otherwise, we set $A_{ij}$ to an integer in $\{-M,\hdots,M\}$ with equal probability.
\item We construct $b$ by first generating a binary solution $\hat{x}$ uniformly at random, and then letting $b = A \hat{x}$.
\end{enumerate}
For each $n \in \{ 10,12,14,16 \}$, we generate 100 instances and discard the ones with $\frac{z^{LP}}{z^I} \leq 2$, after which we obtain 75, 83, 84, 84 instances, respectively. The results of this experiment is given in Figure \ref{fig:CGplot}, where each circle corresponds to a single instance. We observe that for the majority of the instances, the ratio $z^{1\C}/z^\C$ is significantly larger than 2. The arithmetic and geometric means of the ratio for different values of $n$ are $11.46,12.80,13.92,15.16$ and $5.68,7.46,8.51,9.94$, respectively.

\begin{figure}[h]
\centering
\includegraphics{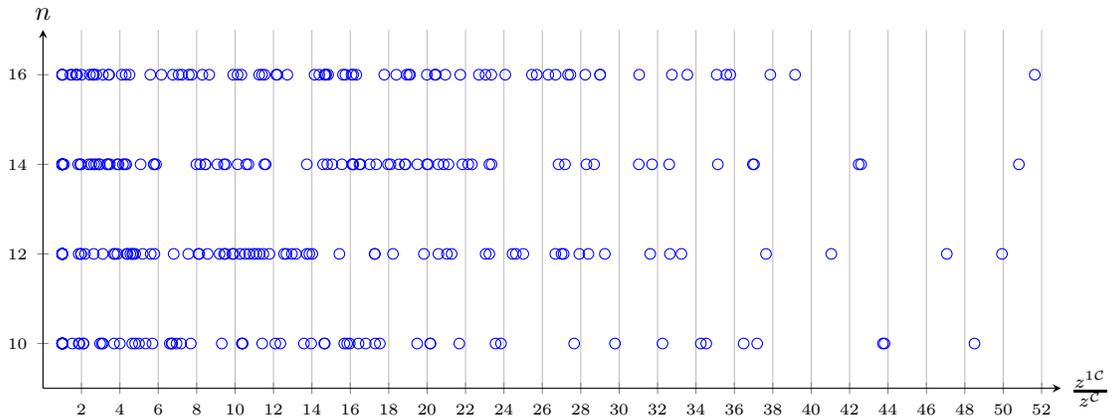}
\caption{Multiplicative gap between 1-row CG closure and CG closure of randomly generated instances}
\label{fig:CGplot}
\end{figure}

\paragraph{Market split instances.} This type of instances, also known as market share instances, are formulated in \cite{cornuejols:da:1999} and consist of a class of small 0-1 IPs that are very difficult for branch-and-cut solvers. We use the following parameters:
\begin{enumerate}
\item We consider instances with $m=2$ equality constraints and  $n = 10(m-1)$ variables.
\item We take $u_j = 1$ for all $j \in [n]$.
\item For any $i \in [m], j \in [n]$, we let $A_{ij}$ to be an integer drawn uniformly from $\{0,\hdots,D-1\}$, where $D = 50$.
\item We set $b_i = \lfloor \sum_{j=1}^n A_{ij}/2 \rfloor$, for all $i \in [m]$.
\end{enumerate}
It has been argued in \cite{cornuejols:da:1999} that most of those instances are infeasible. In this setting, we want to check how often 1-row CG closure detects infeasibility in comparison to the regular CG closure. We generated 100 instances, among which only 10 were feasible. 
The results of this experiment are presented in the table below. 
\begin{table}[h]
\centering
  \begin{tabular}{ c c c }
    \hline 
    1-row CG & CG & \# instances \\
    \hline
    feasible & infeasible & 50 \\ 
    infeasible & infeasible & 40 \\ 
    feasible & feasible & 10 \\
    \hline
  \end{tabular}
\end{table}
As seen in the table, aggregation-based cuts are significantly better than 1-row cuts in proving infeasibility.


\section{Some open questions}
\label{sec:Discussion}

Many interesting open questions can be pursued as future research. The first one is a structural question: Is the aggregation closure a polyhedron? For covering and packing IPs, if the constraint matrix $A$ defining the LP relaxation is \emph{dense} (i.e., every entry of $A$ is positive), then we can show that the closure is polyhedron, see Appendix \ref{sec:appendix}. However, the question remains open for the general case. Another question is to understand if we can restrict the set of aggregation multipliers to generate cuts for general IPs based on the sign pattern of the constraint matrix $A$ to approximate the overall aggregation closure well. 


\section{Packing problems}
\label{sec:Packing}

	In this section we present the proof for the statements regarding packing problems. 

	A crucial tool to analyze the infinite intersections 	arising in the aggregation closures is the following alternative characterization of $\alpha$-approximation, which is well-known in the covering polyhedral case \cite{goe95}; a quick proof is presented in Appendix \ref{app:charApproxPacking}.

	\begin{proposition} \label{prop:approxPack}
		Consider two packing sets $U \supseteq V$ in $\R^n$. Then $U$ is an $\alpha$-approximation of $V$ if and only if $U \subseteq \alpha V$.
	\end{proposition}
The usefulness of this characterization comes from the following: since set containment is preserved under intersections, if $U_i$ is an $\alpha$-approximation of $V_i$ for all $i \in I$ (an arbitrary set), then $\bigcap_{i \in I} U_i \subseteq \bigcap_{i \in I} \alpha V_i = \alpha \bigcap_{i \in I} V_i$ and thus $\bigcap_{i \in I} U_i$ is an $\alpha$-approximation of $\bigcap_{i \in I} V_i$. The equality in this argument follows from this simple observation (with $\phi(S) = \alpha S$).

\begin{observation}\label{obs:bijection}
Let $\phi:\mathbb{R}^n \rightarrow{R}^n$ be a bijective map, let $\{S^i\}_{i \in I}$ be a collection of subsets in $\mathbb{R}^n$ and let $\phi(S):= \{ \phi(x) \,|\, x \in S\}$. Then $\phi\left(\bigcap_{i \in I} S^i\right) = \bigcap_{i \in I} \phi(S^i)$.
\end{observation}

We also note the following. 
\begin{proposition} \label{prop:inthullPack}
Let $Q$ be a packing set. Then, $Q^I$ is also a packing set.
\end{proposition}
The proof of Proposition \ref{prop:inthullPack} is given in Appendix \ref{app:packinginthull}.


\subsection{Proof of Theorem \ref{thm:pack}}

	\paragraph{Upper bound.} 
	We show the first part of the theorem. That is, consider a non-negative objective function $c \in \R^n_+$; we need to show that $z^{1\mathcal{M}} \le 2z^\mathcal{M}$
	for the closures $\mathcal{M} \in \{ \A, \C \}$.  
	
	Let $Q = \{x \in \R^n_+ \mid Ax \le b \}$ be a packing polyhedron.
	As mentioned in the introduction, a main handle to prove the result is to look at a \emph{pre-processed LP} of $Q$, which sets to 0 variables that have too large left-hand-side coefficients. More precisely, let $S$ be the set of indices $j$ where $A_{ij} > b_i$ for some $i$; the pre-processed LP is then $$\LPstar(Q) := \{ x \in \mathbb{R}^n_{+} \mid Ax \le b, ~x_j = 0\ \forall j \in S\}.$$
		
	As seen in Figure \ref{fig:packProofSch}, we prove $z^{1\mathcal{M}} \le 2z^\mathcal{M}$ by showing the following chain of inequalities: 
	$$z^{1\A} \leq z^{1\C} \leq z^{\LPstar} \leq 2 z^\A \leq 2 z^\C.$$
	The first inequality $z^{1A} \leq z^{1C}$ follows trivially by definition. The inequality $z^{A} \leq z^C$ is also obvious. It remains to show that $z^{1\C} \leq z^{\LPstar} \leq 2 z^\A$.
	
	We first show that the 1-row CG closure already captures the power of the pre-processed LP.

\begin{proposition}
\label{prop:pack_1CandLP}
$z^{1\C} \leq z^{\LPstar}$.
\end{proposition}
\begin{proof}
Suppose $A_{ij} > b_i$ for some $i \in [m]$, $j \in [n]$; it is sufficient to show that the inequality $x_j \le 0$ is valid for $1\C(Q)$. Consider the $i^{\text{th}}$ constraint and the following CG cut generated from it:
$$ \sum_{k = 1}^n  \left\lfloor\frac{A_{ik}}{A_{ij}}\right\rfloor x_j \leq \left\lfloor \frac{b_i}{A_{ij}}\right\rfloor = 0.$$
Observe that, over $\R^n_+$, this inequality dominates the inequality $x_j \leq 0$.
\end{proof}

We now show that the pre-processed LP gives a 2-approximation to the aggregation closure (and hence to the CG closure). 

\begin{proposition}
\label{prop:pack2}
$z^{\LPstar} \leq 2 z^{\A}$.
\end{proposition}
\begin{proof}
We begin with a preliminary result. 

\paragraph{Claim 1} Consider a single-constraint packing polyhedron $P^1 := \{x \in \mathbb{R}_+^n \,|\, a^{\top}x \le b_0\}$ and the related polyhedron $P^2 := \{x \in \mathbb{R}_+^n \,|\, a^{\top}x \le b_0, \ x_j = 0 \ \forall j \in S\}$ where $S \supseteq \{j \,|\, a_j > b_0\}$. Then ${P^2} \subseteq 2 (P^1)^{I}$.
\\ \\ \emph{Proof.} 
If $S = [n]$, then the result is trivially true. Otherwise, consider a cost function $c \in \mathbb{R}^n_+$ and let $x^*$ be a maximizer of $c$ over ${P}^2$. Notice $x^*$ simply sets the coordinate not in $S$ with largest ratio $c_j/a_j$ to value $b_0/a_j$. So rounding down $x^*$ gives a point in $(P^1)^{I}$ with $c$-value at least half that of $x^*$. This implies that $P^2$ is a 2-approximation for $(P^1)^I$, and so Proposition \ref{prop:approxPack} gives the desired inclusion $P^2 \subseteq 2 (P^1)^I$, which follows from the fact that $(P^1)^I$ is a packing polyhedron (by Proposition \ref{prop:inthullPack}). $\diamond$
\\ \\ 
\indent
Let $S := \{j\,|\, A_{ij} > b_i \textup{ for some } i\in [m], j \in [n]\}$. For $\lambda \in \R_+^m$, let $Q_{\lambda} = \{ x \in \R^n_+ \mid \lambda^\top Ax \leq \lambda^\top b \}$ and  
$(\LPstar(Q))_{\lambda} = \{x \in \mathbb{R}_+^n \,|\, \lambda^{\top} {A} x \le \lambda^{\top} b, \ x_j = 0, \forall j \in S\}$. Using Claim 1 we have that for any $\lambda \in \mathbb{R}^m_{+}$ 
\begin{eqnarray}\label{eq:mainpack}
\LPstar(Q) \subseteq (\LPstar(Q))_{\lambda} \subseteq 2\left({Q}_{\lambda}\right)^I.
\end{eqnarray}
Taking intersection over all $\lambda \in \R^n_+$ we obtain that 
\begin{align*}
	\LPstar(Q) \subseteq \bigcap_{\lambda \in \R^m_+}2\left({Q}_{\lambda}\right)^I  = 2\bigcap_{\lambda \in \R^m_+}\left({Q}_{\lambda}\right)^I = 2 \A(Q), 
\end{align*}
where the first equation uses Observation \ref{obs:bijection}. Then from Proposition \ref{prop:approxPack} we have that $\LPstar(Q)$ is a 2-approximation for $\A(Q)$.
\end{proof}


\paragraph{Tight instances.} To prove the second part of the theorem, it suffices to show that there are instances where the 1-row closure, which is the strongest 1-row closure we consider, is at most roughly a 2-approximation of the CG closure, the weakest closure we consider.

\begin{proposition}
\label{prop:pack_TightEx}
For every $\epsilon>0$ there exists an instance where $\frac{z^{1\A}}{z^{\C}} \geq 2 - \epsilon$.
\end{proposition}
\begin{proof}
Consider the following family of packing IPs
\begin{align}
\text{maximize} \quad & x_1 +  x_2 \nonumber \\
\text{subject to} \quad & x_1 + Mx_2 \leq M \label{packingtightexA} \\
& Mx_1 + x_2 \leq M \label{packingtightexB} \\
& x \ge 0 \label{c} \\
& x \in \mathbb{Z}^2, \nonumber
\end{align}
where $M$ is an integer with $M \ge 1$.

We show that $\lim_{M \rightarrow \infty} \frac{z^{1\A}}{z^{\C}} \rightarrow 2$.
Observe that the set $\{x \in \mathbb{R}^2_{+}\,|\,  x_1 + Mx_2 \leq M\}$ and the set $\{x \in \mathbb{R}^2_{+}\,|\,  Mx_1 + x_2 \leq M\}$ are integral. Therefore, $z^{1\A} = z^{LP}$, 
or equivalently $z^{1\A} = \frac{2M}{M + 1}$. 

On the other hand, since \eqref{packingtightexA} and \eqref{packingtightexB} imply that the inequality $x_1 + x_2 \leq \frac{2M}{M + 1}$ is valid for $Q$, we have that $x_1 + x_2 \leq 1$ is a valid CG cut for $Q$. Therefore, we obtain $\C(Q) \subseteq \{(x_1, x_2) \in \mathbb{R}^2_{+}\,|\, x_1 + x_2 \leq 1 \}$. Thus $z^{\C} = 1$.
\end{proof}

\subsection{Proof of Theorem \ref{thm:noncoverpack}}

Let $k \in \mathbb{Z}$ and $k \geq 2$ and consider the following IP: 
\begin{eqnarray}
&\textup{max}& x_1 + x_2 \nonumber \\
&\textup{s.t.}& k^2 x_1 - (k-1) x_2 \leq k^2 \label{con1}\\
&& -kx_1 + x_2 \leq -k + 1 \label{con2}\\
&& x_1, x_2\geq 0. \label{con3}
\end{eqnarray}
	To prove the theorem, it suffices to show that as $k$ goes to infinity, the ratios $\frac{z^{1\C}}{z^{\C}}$ and $\frac{z^{1\A}}{z^{\A}}$ also go to infinity. In fact, since $z^{1\A} \le z^{1\C}$ and $z^{\A} \le z^{\C}$, we just need to show that  $\frac{z^{1\A}}{z^{\C}} \rightarrow \infty$.  
	
\paragraph{1-row closure.} We verify that the point $(2 - \frac{1}{k}, k)$ belongs to the original 1-row cut closure. Consider the following cases:
\begin{enumerate}
\item Integer hull of (\ref{con1}) and (\ref{con3}): The points $(1,0)$ and $(k, k^2)$ are valid integer points and $(2 - \frac{1}{k}, k)$ is a convex combination of these points. 
\item Integer hull of (\ref{con2}) and (\ref{con3}): The points $(1,1)$ and $(2, k + 1)$ are valid integer points and $(2 - \frac{1}{k}, k)$ is a convex combination of these points.  
\end{enumerate}

Since $(2-\frac{1}{k},k)$ belongs to the 1-row closure, by inspecting its objective value we obtain that $z^{1\A} \geq 2 + k - \frac{1}{k}$. 
\paragraph{CG closure.} To upper bound the optimal value of the CG closure, we explicitly construct one CG cut. Consider the aggregation of the LP inequalities 
$\frac{1}{k}\times$(\ref{con1}) $+$ $\frac{k-1}{k}\times(\ref{con2}) \equiv x_1 \leq 2-\frac{1}{k}$, which gives the CG cut $x_1 \leq 1.$ 

We can compute an upper bound on $z^{\C}$ by computing the optimal value subject to this CG cut and \eqref{con2}, namely $\textup{max}\{ x_1 + x_2 \,|\, x_1 \leq 1, \ -k x_1 + x_2 \leq -k + 1, x_1 \geq 0, x_2 \geq 0\} = 2$: $(1,1)$ is a feasible primal solution and $(k+1, 1)$ is a dual feasible solution with same objective function value. Therefore, we have $z^{\C} \leq 2$.

\medskip

Putting these bounds together obtain that $\frac{z^{1\A}}{z^{\C}} = \frac{k}{2} + 1 - \frac{1}{2k}$, which goes to infinity as $k \rightarrow \infty$. This concludes the proof of the theorem. \hfill \qed
\subsection{Proof of Theorem~\ref{thm:packrankgen}}

A key result we use is given in Proposition \ref{prop:packingDf} below, which provides a bound on the integrality gap as a function of the number of inequalities. Note that since we do not make any assumptions on the coefficients of the constraint matrix of the packing polyhedron, we obtain better coefficients than those obtainable by using randomized rounding arguments; see for example \cite{srinivasan1999improved}. 

\begin{proposition}
\label{prop:packingDf}
Consider a packing IP of the following form $\max \{c^{\top}x \mid Dx \leq f, \ x \in \mathbb{Z}_{+}^n\}$ where $D$ is a $k \times n$ non-negative matrix such that $D_{ij} \leq f_i$ for all $i \in [k]$ and $j \in [n]$ and $c \in \mathbb{R}^n_{+}$. Then $z^{LP} \leq (k+1) z^I$. 
\end{proposition}

\begin{proof}
If the LP has unbounded value, then the IP also has unbounded value \cite{NemWolBook}, and there is nothing to prove. 

Assume the LP has bounded value, and let $x^{LP}$ be an optimal solution of the LP. Let 
$$x^{LP} = \hat{x} + x^F,$$
where $\hat{x}$ is obtained by rounding down $x^{LP}$ componentwise. Then, $\hat{x}$ belongs to the feasible region of the packing problem, and hence $z^I \geq c^\top \hat{x}$.

Let $c_{max} \in \text{argmax}_{j} \{ c_j \}$. Since $e_j$ for all $j \in [n]$ belongs to the feasible region, $z^I \geq c_{max}$. Thus, $z^I \geq \max \{ c^\top \hat{x},c_{max}\}$.

Now, observe that 
\begin{align*}
\frac{z^{LP}}{z^I} & \leq \frac{c^\top \hat{x}+c^\top x^F}{\max \{ c^\top \hat{x},c_{max}\}} \\
& \leq 1 + \frac{c^\top x^F}{\max \{ c^\top \hat{x},c_{max}\}}.
\end{align*}

Since there are $k$ constraints $D^i x \le f_i$, at most $k$ components of $x^{LP}$ can be non-zero. In other words, at most $k$ components of $x^F$ can be non-zero. Also, each entry of $x^F$ is strictly less than 1. Hence, $c^\top x^F \leq c_{max} k$, and therefore 
$$\frac{z^{LP}}{z^I} \leq 1+ \frac{c_{max} k}{\max \{ c^\top \hat{x},c_{max} \} } \leq 1+ k.$$
\end{proof}
\paragraph{Lower bound on rank.} We actually prove Theorem \ref{thm:packrankgen} for the more general case of packing sets containing all the basis vectors $e_j$'s; notice that for a packing polyhedron $Q = \{x \in \R^n_+ \mid Ax \le b\}$, containing all basis vectors $e_j$'s is equivalent to the condition $A_{ij} \le b_i$ for all $i,j$.

So let $Q$ be a non-empty packing set containing all the basis vectors $e_j$'s. Given a matrix $(D,f) \in \R^{k \times n} \times \R^k$, we say that it is a \emph{\kvi} for $Q$ if $(D,f)$ is non-negative and the inequalities $D^i x \le f_i$ are valid for $Q$. We denote the polyhedral outer-approximation $\{ x \in \R^n_+ \mid D x \leq f \}$ of $Q$ by $P_{(D,f)}$. Then by definition 
\begin{align}
	\A_k(Q) = \bigcap_{(D,f)\textrm{ is a \kvi \ for $Q$}} (P_{(D,f)})^I. \label{eq:kviIntersection}
\end{align}

Let $Q^\ell$ be the $\ell^{\text{th}}$ $k$-aggregation closure of $Q$. 
\paragraph{Claim 1} 
$Q^\ell \subseteq (k+1)\, Q^{\ell+1}$. 
\\ \\ \emph{Proof.} 
Consider a \kvi \ $(D,f)$ for $Q^\ell$. Clearly $P_{(D,f)}$ is a packing polyhedron, and since $P_{(D,f)} \supseteq \A_k(Q^\ell) \supseteq Q^I$, all basis vectors $e_j$ belong to $P_{(D,f)}$. Therefore, $D_{ij} \leq f_i$ for all $i \in [k], ~ j \in [n]$, and so by Proposition \ref{prop:packingDf} we obtain that $P_{(D,f)}$ is a $(k+1)$-approximation of $(P_{(D,f)})^I$. Hence, Proposition \ref{prop:approxPack} gives 
$$Q^\ell \subseteq P_{(D, f)} \subseteq (k+1) (P_{(D, f)})^I.$$
So taking intersection over all \kvi's and using Observation~\ref{obs:bijection}, we have that
\begin{align*}
\begin{split}
Q^\ell \ \subseteq & \  \bigcap_{(D,f) \textup{ is \kvi \ for } Q^\ell} (k+1) (P_{(D, f)})^I  \\
= & \ (k+1) \bigcap_{(D,f) \textup{ is \kvi \ for } Q^\ell} (P_{(D, f)})^I \\
= & (k+1)\, Q^{\ell+1},
\end{split}
\end{align*}
where the last equality follows from \eqref{eq:kviIntersection}. This concludes the proof.
~$\diamond$
\\  \\ 
\indent
Finally, suppose the rank of $k$-aggregation is $t$ and let $z^{i}$ be the optimal objective function value over the $i^{\text{th}}$ closure. Since all of these closures are packing sets, Claim 1 and Proposition \ref{prop:approxPack} guarantee that $z^i \le (k+1) z^{i+1}$. Therefore,
\begin{eqnarray*}
\frac{z^{LP}}{z^I} = \frac{z^{LP}}{z^1} \frac{z^1}{z^2}\dots\frac{z^{t-1}}{z^{t}} \leq (k+1)^{t}.
\end{eqnarray*}
This implies the inequality
\begin{eqnarray*}
t = \rank_{\A_k}(Q) \geq \left\lceil\frac{\textup{log}_2\left( \frac{z^{LP}}{z^I}\right)}{\textup{log}_2(k+1)}\right\rceil,
\end{eqnarray*}
which is the required result.

\paragraph{Tight example.} We now show that 
there is a packing integer set $Q$ with $\rank_{A_1}(Q) \le O\left(\textup{log}_2\left( \frac{z^{LP}}{z^{I}}\right)\right)$.
Let $K_n$ be a complete graph with node set $[n]$, and let $Q$ be the standard edge-relaxation of the stable set polytope:
\begin{align*}
FSTAB(K_n) = \{ x \in \R_+^{n} \,|\, x_i + x_j \le 1 \ \forall i,j \in [n], \ i < j\}.
\end{align*}
If our objective is to maximize $\sum_{v \in [n]} x_v$, then we obtain $z^{I}=1$ and $z^{LP}=n/2$ because the optimal vertex of $FSTAB(K_n)$ is the vector with all entries equal to $1/2$.
Consider now the clique inequality $\sum_{v \in [n]} x_v \le 1$, which defines a facet of the stable set polytope.
We only need to show that the CG rank of the clique inequality is upper bounded by $O\left(\textup{log}_2\left( \frac{z^{LP}}{z^{I}}\right) \right)= \lceil \log_2 (n-1) \rceil$.
The latter is a well-known fact \cite{hartmann}.
\qed


\section{Proofs for covering problems}
\label{sec:Covering}

	We now provide additional definitions and proofs of the statements presented in the introduction regarding covering problems: in Subsection \ref{sec:proofCover} we prove Theorem \ref{thm:cover}, the main result of this section, and in Subsection \ref{subsec:CoveringRankProof} we prove Theorem \ref{thm:rankCover}. Before proving these results we need to develop some general results concerning covering sets with bounds. 

	
	\subsection{Properties of covering sets with bounds}
	\label{subsec:CoveringProperties}

	We start by showing that adding non-negative directions to a covering polyhedron with bounds still leaves it as a covering polyhedron (possibly with bounds); in fact, adding all the non-negative directions is a natural way of removing the upper bounds. 
	
	Given a covering polyhedron with bounds of the form $P = \{ x \in \R^n_+ \mid Ax \geq b,~ x \leq u \}$ with $A,b \ge 0$, we refer to $Ax \geq b$ as the covering inequalities of $P$.

\begin{proposition} \label{prop:coverUpward}
	Consider a covering polyhedron with bounds $P = \{ x \in \R^n_+ \mid Ax \ge b, \ x \le u\}$. Then, for any subset $\{e_j\}_{j \in J}$ of the canonical vectors we have that $P + \cone(\{e_j\}_{j \in J})$ is a covering polyhedron with bounds. In particular, $P + \R^n_+$ is a covering polyhedron.
	
	Moreover, each covering inequality of $P + \cone(\{e_j\}_{j \in J})$ is a conic combination of one covering inequality of $P$ with the bounds $x_j \le u_j$ for $j \in J$.
	\end{proposition}

\begin{proof}
	Notice it suffices to show that for a single $e_j$, $P + \cone(e_j)$ is a covering polyhedron with bounds (the general statement follows by the repeated application of this result).
	
	So consider one such $e_j$. For every inequality $A^i x \ge b_i$ of the system $Ax \ge b$, let $\hat{A}^i x \ge \hat{b}_i$ be the sum of $A^i x \ge b_i$ and $-A_{ij} x_j \ge -A_{ij} u_j$. Note that $\hat{A}^i \ge 0$ and $\hat{A}_{ij} = 0$.
Let $\hat Ax \ge \hat b$ be the system comprising all such inequalities $\hat{A}^i x \ge \hat{b}_i$.
Let $\hat u$ be the vector obtained from $u$ by replacing $u_j$ with $\infty$.
We define the covering polyhedron with bounds $\hat P = \{x \in \mathbb{R}^n_+ \,|\, Ax \ge b, \ \hat Ax \ge \hat b, \ x \le \hat u\}$.
By construction, each covering inequality of $\hat P$ is a conic combination of one covering inequality of $P$ with the bound $x_j \le u_j$.
In the remainder of the proof we show $P + \cone(e_j) = \hat P$.

Since each inequality valid for $\hat{P}$ is also valid for $P$ and the recession cone of $\hat{P}$ contains $e_j$ ($A$ and $\hat{A}$ are non-negative and $\hat{u}_j = \infty$), we have $P + \cone(e_j) \subseteq \hat P$.

We now show the reverse inclusion $P + \cone(e_j) \supseteq \hat P$. Let $c^\top x \ge \delta$ be a valid inequality for $P + \cone(e_j)$. Equivalently, $c^\top  x \ge \delta$ is a valid inequality for $P$ with $c_j \ge 0$.
As a consequence, there exist nonnegative multipliers $\mu_i, \delta_i, \gamma_i$ such that 
$$c = \sum_{i=1}^m \mu_i A^i + \sum_{i=1}^n \delta_i e^i - \sum_{i=1}^n \gamma_i e^i \quad \text{and} \quad \delta \leq \delta_0:= \sum_{i=1}^m \mu_i b_i - \sum_{i=1}^n \gamma_i u_i.$$
Without loss of generality we can assume that at least one among $\delta_j$ and $\gamma_j$ equals zero.
In the latter case, the inequality $c^\top x \ge \delta$ is trivially valid for $\hat P$, thus we now assume $\delta_j = 0$ and $\gamma_j > 0$.
Since $c_j \ge 0$, we have $c_j = \sum_{i=1}^m \mu_i a_{ij} - \gamma_j \ge 0$.

Let $k \in \{1,\dots,m\}$ be the smallest index such that $\sum_{i=1}^k \mu_i A_{ij} \ge \gamma_j$.
In this way $\gamma_j - \sum_{i=1}^{k-1} \mu_i A_{ij} > 0$.
This allows us to define non-negative multipliers $\lambda_i$, $\lambda'_i$, for $i=1,\dots,m$:
\begin{align*}
\lambda_i = 
\begin{cases} 
0  \\ 
\mu_k- \frac{\gamma_j - \sum_{i=1}^{k-1} \mu_i A_{ij}}{A_{kj}} \\ 
\mu_i  \\ 
\end{cases},
\ 
\lambda'_i = 
\begin{cases} 
\mu_i &\mbox{if } i=1,\dots,k-1 \\ 
\frac{\gamma_j - \sum_{i=1}^{k-1} \mu_i A_{ij}}{A_{kj}} &\mbox{if } i=k \\ 
0 &\mbox{if } i=k+1,\dots,m. \\ 
\end{cases} 
\end{align*}

It can be verified that:
$$c = \sum_{i=1}^m \lambda_i A^i + \sum_{i=1}^m \lambda'_i \hat A^i  + \sum_{i=1}^n \delta_i e^i - \sum_{\substack{i=1 \\ i \neq j}}^n \gamma_i e^i \ \ \text{and} \ \ \delta_0 = \sum_{i=1}^m \lambda_i b_i + \sum_{i=1}^m \lambda'_i \hat b_i - \sum_{\substack{i=1 \\ i \neq j}}^n \gamma_i u_i.$$
This implies that $c^\top x \ge \delta$ is valid for $\hat P$. 

This shows that every valid inequality for $P + \cone(e_j)$ is valid for $\hat P$, hence $P + \cone(e_j) \supseteq \hat{P}$ and we conclude the proof of the proposition. 
\end{proof}

	Next, we show that the integer hull of a covering polyhedron with bounds is also a covering polyhedron with bounds. 
	
	\begin{proposition}\label{prop:coverIntHull}
		Let $Q = \{x  \in \Z^n_+ \mid Ax \ge b, \ x \le u\}$ be a non-empty covering polyhedron with bounds (recall that $u$ is integral or infinite). Then its integer hull $Q^I$ is a covering polyhedron with bounds. Moreover, $Q^I$ has the same upper bounds as $Q$, namely $Q^I = \{x \in \R^n_+ \mid A'x \ge b', \ x \le u\}$ for some $(A',b')$.
	\end{proposition}
	
	\begin{proof}
		We assume that $Q^I$ is non-empty, otherwise the result can be easily verified. Let $\pi^\top  x \ge \pi_0$ be a facet-defining inequality for $Q^I$. It suffices to show that either $(\pi, \pi_0) \ge 0$, or that this inequality is equivalent to an upper bound constraint $x_j \le u_j$ for some $j$ and $u_j$.
		
		First, suppose $\pi \ge 0$. Then $\pi_0$ must be non-negative, since otherwise the fact that $Q^I \subseteq \R^n_+$ would imply that the face of $Q^I$ induced by $\pi^\top  x \ge \pi_0$ is empty, contradicting that it is a facet. 
		
		Now consider the case where $\pi$ has at least one negative coordinate, say $\pi_j < 0$. If all other components of $\pi$ are equal to 0, then $\pi^\top  x \ge \pi_0$ is equivalent to an upper bound constraint: $$\pi^\top  x \ge \pi_0 ~\equiv~ \pi_j x_j \ge \pi_0 ~\equiv~ x_j \le \frac{-\pi_0}{\pi_j},$$
		where the sign/sense reversal in the last equivalence happens because $\pi_j$ is negative. Thus, to conclude the proof it suffices to consider the case where $\pi$ has support of size at least 2.
		
		We show that this case actually leads to a contradiction. The idea is to use the following property that can be immediately verified: if $\bar{x}, \bar{y}$ are \emph{integer} points in $Q^I$, then the point $\bar{z}$ obtained by taking $\bar{x}$ and replacing its $j^{\text{th}}$ component by $\max\{\bar{x}_j, \bar{y}_j\}$ also belongs to $Q^I$. Moreover, if $\bar{y}_j > \bar{x}_j$ we have that $\pi^\top  \bar{z} < \pi^\top \bar{x}$; we will use this to contradict the validity of $\pi^\top  x \ge \pi_0$. 
		
		To make this precise, since $\pi^\top  x \ge \pi_0$ is facet-defining, let $\bar{x}^1, \ldots, \bar{x}^n$ be affinely independent integer points in $Q^I$ that satisfy the  equality $\pi^\top  x = \pi_0$. Let $M = \max_i \bar{x}^i_j$ be the maximum value in the $j^{\text{th}}$ coordinate of these points. Observe that at least one of the points $\bar{x}^i$ has the $j^{\text{th}}$ coordinate strictly smaller than $M$: otherwise all points $\bar{x}^i$ would satisfy the linearly independent inequalities $\pi^\top  x = \pi_0$ and $x^i_j = M$ (the linear independence comes from the fact $\pi$ has support of size at least 2) and thus would lie in an $(n-2)$-dimensional space, contradicting that they are $n$ affinely independent points. 
		
		Thus, without loss of generality assume that $\bar{x}^1_j = M > \bar{x}^2_j$. Construct the point $\bar{z}$ by taking the vector $\bar{x}^2$ and replacing its $j^{\text{th}}$ coordinate by $\max\{\bar{x}^1_j, \bar{x}^2_j\} = \bar{x}^1_j$. As mentioned earlier, $\bar{z}$ belongs to $Q^I$ but $$\pi^\top  \bar{z} < \pi^\top  \bar{x}^2 = \pi_0,$$ thus contradicting the validity of $\pi^\top  x \ge \pi_0$. This concludes the proof that $Q^I$ is a covering set.
		
		To see that the upper bounds in $Q^I$ are the same as those in $Q$, let $Q^I = \{x \in \R^n_+ \mid A'x \ge b', \ x \le u'\}$ be a covering-with-bounds description of this set with minimal $u'$ (i.e. there is no other valid upper bound that is pointwise smaller than $u'$). Recall that $u$ is the vector of upper bounds in $Q$, which is an integral vector. Since $Q^I \subseteq Q \subseteq [0,u]$, the minimality of $u'$ guarantees that $u' \le u$. But since $Q$ is non-empty, it contains the point $u$, and so does the integer hull $Q^I$; thus, $u' \ge u$. This concludes the proof. 	 
	\end{proof}

	We also remark the following equivalent definition of $\alpha$-approximation, similar to that for the packing case; the first part of the statement follows directly from the definition of $\alpha$-approximation, and the second follows from Proposition \ref{prop:coverUpward} combined with Lemma 23 of \cite{molinaro2013understanding}.
	
	\begin{proposition} \label{prop:blowupCover}
		Consider two covering sets $U \supseteq V$. 
		Then $U$ is an $\alpha$-approximation of $V$ iff $U + \R^n_+$ is an $\alpha$-approximation of $V + \R^n_+$. Moreover, this happens iff $\frac{1}{\alpha}(U + \R^n_+) \subseteq (V + \R^n_+)$.
	\end{proposition}
	
	Finally, we need the following property, which states that for covering polyhedra with the \emph{same} upper bounds we can commute adding $\R^n_+$ and taking intersections.
	
	\begin{proposition} \label{prop:commCover}
		Let $\{Q^i\}_{i \in I}$ be a (possibly infinite) family of covering polyhedra with bounds such that all upper bounds are the same, namely $Q^i = \{x \in \R^n_+ \mid G(i) x \geq g(i), \ x \le u\}$ for all $i \in I$ (where $G(i) \in \R_+^{m_i \times n},~ g(i) \in \R_+^{m_i}$). Then $$\bigcap_{i \in I} (Q^i + \R^n_+) = \bigg( \bigcap_{i \in I} Q^i \bigg) + \R^n_+.$$
	\end{proposition}
	
	\begin{proof}
		The direction ``$\supseteq$'' is straightforward, so we prove the direction ``$\subseteq$''. Consider a point $x \in \bigcap_{i \in I} (Q^i + \R^n_+)$, so we can write $x = q^i + r^i$ for $q^i \in Q^i$ and $r^i \ge 0$. The idea is that if we push all the $q^i$'s coordinates as high as possible (correcting appropriately the $r^i$'s) we can actually get the same point in all the $Q^i$'s.
		
		More explicitly, define the point $q \in \R^n$ as follows: if $x_j$ is at most the upper bound $u_j$, set $q_j = x_j$, else set $q_j = u_j$ (so $q = \min\{x, u\}$). We claim that $q$ belongs to $Q^i$ for all $i$. First, since $q^i \le x$ and $q^i \le u$, we have that $q^i \le q$; therefore, since $q^i$ satisfies the covering constraints of $Q^i$, so does $q$. Moreover, $q \le u$, so $q$ also satisfies the upper bound constraints of $Q^i$; thus $q \in Q^i$. We then get that the point $q + (x - q)$ belongs to $( \bigcap_{i \in I} Q^i) + \R^n_+$. This shows the desired inclusion and concludes the proof. 
	\end{proof}
	

	We can now start the proof of Theorem \ref{thm:cover}.


	\subsection{Proof of Theorem \ref{thm:cover}} \label{sec:proofCover}

A central object for our proof are the \emph{knapsack-cover} inequalities~\cite{wolsey:1975}. Consider a covering polyhedron with bounds $Q = \{x \in \R^n_+ \mid Ax \ge b, ~x \le u\}$. A knapsack-cover (KC) inequality is generated as follows: Consider a single row $A^i x \ge b_i$ of this problem; given a subset $S \subseteq [n]$ of the variables, the corresponding KC inequality is given by $\sum_{j \notin S} \tilde{A}_{ij} x_j \ge b_i - \sum_{j \in S} u_j A_{ij}$, where $\tilde{A}_{ij} = \min\{A_{ij}, b_i - \sum_{j \in S} u_j A_{ij}\}$. Notice that the KC inequalities are indeed valid for $Q$. Again, we use $KC(Q)$ to denote the \emph{KC closure} (namely the set obtained by adding all the KC inequalities to the linear relaxation of $Q$), and for a given objective function we use $z^{KC}$ to denote the optimal value of optimizing this function over $KC(Q)$.

	We break down the proof of Theorem \ref{thm:cover} by first comparing $1\A(Q)$ versus $\A(Q)$; we then compare $1\C(Q)$ versus $\C(Q)$, which is significantly more involved.

\subsubsection{Proof for aggregation closure}

	\paragraph{Upper bound.} Observe that the 1-row closure is at least as strong as the $KC$ closure by construction of the KC inequalities. We need the following result, which states that for a 1-row covering polyhedron with bounds, the KC closure is a 2-approximation of the integer hull.

\begin{theorem}[\cite{carr2000strengthening}] \label{thm:carr}
Consider a 1-row covering polyhedron with bounds $Q = \{x \in \Z^n_+ \mid ax \ge b, \ x \le u\}$. Then the KC closure $KC(Q)$ is a 2-approximation of the integer hull $Q^I$.
\end{theorem}

Since the aggregation closure is the intersection of the integer hull of multiple 1-row covering polyhedra, we leverage the theorem above to show that the KC closure is also a 2-approximation for the aggregation closure of a multi-row covering polyhedron. 

\begin{proposition}
\label{prop:KCvsA}
For every covering polyhedron with bounds $Q$ we have that the KC closure $KC(Q)$ is a 2-approximation of the aggregation closure $\A(Q)$.
\end{proposition}	

\begin{proof}	
Let $Q = \{ x \in \R^n_+ \mid Ax \ge b, \ x \le u\}$ and consider $Q_{\lambda} := \{  x \in \R^n_+ \mid \lambda^\top Ax \ge \lambda^\top b, \ x \le u \}$ for some $\lambda \in \R^n_+$. First we connect the KC closure of $Q$ with the KC closure of the 1-row covering set $Q_\lambda$, proving the intuitive fact that $KC(Q) \subseteq KC(Q_\lambda)$.

For that, consider a KC inequality 
$$kc := \{\sum_{j \notin S} \widetilde{(\lambda^\top A)_j} x_j \ge \lambda^\top b  -  \sum_{j \in S} (\lambda^\top A)_j u_j \}$$
 for $Q_{\lambda}$ and let $kc_i = \{\sum_{j \notin S} \tilde{A}_{ij} x_j \ge b_i- \sum_{j \in S} A_{ij} u_j\}$ be the corresponding KC inequality for the $i^{\text{th}}$ row of $P$. We show that $kc$ is dominated by the inequalities $kc_i$'s, namely $kc \cap \mathbb{R}^n_+ \supseteq \bigcap_i (kc_i \cap \mathbb{R}^n_+)$. Consider the aggregation $\sum_i \lambda_i kc_i \equiv \sum_{j \notin S} (\sum_i \lambda_i \tilde{A}_{ij}) x_j \ge \lambda^\top b  -  \sum_{j \in S} (\lambda^\top A)_j u_j$; it suffices to show that this dominates $kc$. 
 The RHS's are the same, so it suffices to compare LHS's. Since $\tilde{A}_{ij} = \min\{A_{ij}, b_i - \sum_{j \in S} A_{ij} u_j\}$, it follows that $\sum_i \lambda_i \tilde{A}_{ij} \le \min\{\sum_i \lambda_i A_{ij}, \lambda^\top b - \sum_{j \in S} (\lambda^\top A)_j u_j\}$, which is exactly the $j^{\text{th}}$ entry in the LHS of $kc$. This proves that $KC(Q) \subseteq KC(Q_\lambda)$.
 
 Employing the alternative definition of $\alpha$-approximation given by Proposition \ref{prop:blowupCover} with Theorem \ref{thm:carr}, we get that for every $\lambda$
		\begin{align*}
			(KC(Q) + \R^n_+) \subseteq (KC(Q_{\lambda}) + \R^n_+) \subseteq \frac{1}{2} (Q^I_\lambda + \R^n_+).
		\end{align*}
	From Proposition \ref{prop:coverIntHull} we have that $Q^I_\lambda$ is a covering polyhedron with bounds, and that the upper bounds are that same as in $Q_\lambda$, which are the upper bounds of $Q$. Since all these bounds are the same, we can take intersection of the last displayed inequality over all $\lambda$'s and used the commutativity from Proposition \ref{prop:commCover} to obtain that 
			%
		{
		\begin{align*}
			&(KC(Q) + \R^n_+) \ \subseteq\  \bigcap_{\lambda \in \R^m_+} \frac{1}{2} (Q^I_\lambda + \R^n_+)\  \stackrel{\textrm{Obs} \ref{obs:bijection}}{=} \\
			&\  \stackrel{\textrm{Obs} \ref{obs:bijection}}{=} \ \frac{1}{2}\bigcap_{\lambda \in \R^m_+} (Q^I_\lambda + \R^n_+)\ = \ \frac{1}{2} \bigg( \bigcap_{\lambda \in \R^m_+} Q^I_\lambda \bigg) + \R^n_+.
		\end{align*}}
The right-hand side of this expression is exactly $\frac{1}{2}(\A(Q) + \R^n_+)$, thus employing Proposition \ref{prop:blowupCover} once again we get that the KC closure $KC(Q)$ is a 2-approximation for the aggregation closure $\A(Q)$. This concludes the proof.
\end{proof}

Hence, we obtain that $1\A(Q)$ is a 2-approximation to $\A(Q)$. 


\paragraph{Tight examples.} We next exhibit an instance where $1\A$ is not better than a 2-approximation of $\A$.

\begin{proposition} 
\label{prop:LBAC}
 Let $\epsilon >0$. There exists an instance where $\frac{z^{\A}}{z^{1\A}} \geq 2 - \epsilon$ and $\frac{z^{\C}}{z^{1\C}} \geq 2 - \epsilon$.
\end{proposition}
\begin{proof} Let $n = \textup{min}\{2,\lceil \frac{1}{\epsilon}\rceil\}$.
Consider the following instance 
\begin{eqnarray*}
\begin{array}{rl}
\textup{min}& \displaystyle \sum_{j = 1}^n x_j \\
\textup{s.t.}& x_i + \displaystyle \sum_{j \in [n]\setminus \{i\}} 2x_j \geq 2, \ \forall i \in [n], \\
&x_j \in \mathbb{Z}^n_{+}.
\end{array} \label{eq:LBAC}
\end{eqnarray*}

We show that $\frac{z^{\A}}{z^{1\A}} \geq 2 - \epsilon$ and $\frac{z^{\C}}{z^{1\C}} \geq 2 - \epsilon$ for this instance. 

\begin{enumerate}
\item $z^{1\A} = z^{1\C} = \frac{2n}{2n - 1}$: Observe that the set $\{x \in \mathbb{R}^n_{+}\,|\,  x_i + \sum_{j \in [n]\setminus \{i\}} 2x_j \geq 2\}$ is integral. Thus, $z^{1\A} = z^{1\C}$ and each is equal to the LP relaxation. Adding all these constraints we obtain 
\begin{eqnarray}\label{eq:preCG}
\sum_{j \in [n]} x_j \geq \frac{2n}{2n - 1}
\end{eqnarray}  
On the other hand, setting $x_j = \frac{2}{2n - 1}$, we obtain a feasible solution. Thus, $z^{1\A} = z^{1\C} = \frac{2n}{2n - 1}$.
\item $z^{\A} \geq 2$ and $z^{\C} \geq 2$: Since (\ref{eq:preCG}) is a valid inequality, we obtain the CG cut $\sum_{j \in [n]} x_j \geq 2$. Thus $z^{\C} \geq 2$ and since $z^{\A} \geq z^{\C}$  we obtain $z^{\A} \geq 2$.
\end{enumerate}
Thus, $\frac{z^{\A}}{z^{1\A}} \geq 2 - \frac{1}{n}$ and $\frac{z^{\C}}{z^{1C}} \geq 2 - \frac{1}{n}$; and our choice of $n$ completes the proof. 
\end{proof}


\subsubsection{Proof for CG closure}

We start by considering the case of covering polyhedra \emph{without bounds}.

\begin{proposition}
\label{prop:coverCGnobounds}
	Consider a covering polyhedron without bounds $P$ and a non-negative function $c \in \R^n_+$. Then $z^{1\C} \geq \frac{1}{2} z^{\C}$.
\end{proposition}
\begin{proof}
Consider a covering polyhedron $P = \{x \in \R^n_+ \,|\, Ax \ge b, \ x \ge 0\}$. Let $\C(P) =\{x \,|\, A'x \ge b'\}$ be the CG closure of $P$ (i.e., CG closure is a rational polyhedron~\cite{schrijver1980cutting}).
Without loss of generality we assume that the entries of $A',b'$ are non-negative integers and each CG cut is obtained by rounding up the entries of the constraint $\lambda^\top Ax \ge \lambda^\top b$ for some $\lambda \in \mathbb{R}_+^m$.
Let $(a')^{\top}x \ge \beta'$ be an inequality of the system $A'x \ge b'$.
We show that $(a')^{\top}x \ge \beta'/2$ is a 1-row CG cut for $P$.
The theorem then follows by linear programming duality.

If inequality $(a')^{\top}x \ge \beta'$ is one inequality of the original system $Ax \ge b, \ x \ge 0$ we are done, thus we assume that $a'x \ge \beta'$ is a non-trivial CG inequality for $Ax \ge b, \ x \ge 0$.
This in particular implies $\beta' \ge 1$.
The strict inequality $a'x > \beta' -1$ is valid for $P$.
If $\beta' \ge 2$, then $\beta'-1 \ge \beta'/2$, thus $(a')^{\top}x \ge \beta'/2$ is valid for $P$ and so it is trivially a 1-row CG cut for $P$. Thus we now assume $\beta' =1$.

Let $\lambda \in \R^m_+$ be the vector of multipliers corresponding to the CG cut $(a')^{\top}x \ge \beta'$, i.e., $a' = \lceil \lambda^{\top} A\rceil $, and $\beta' = \lceil \lambda^\top b \rceil$.
Since $\lambda^{\top} b > 0$, there exists $i\in [m]$ with $\lambda_i b_i > 0$.
Then $(a')^{\top}x \ge \beta'$ is implied by the 1-row CG cut $\sum_{j = 1}^n\lceil \lambda_i A_{ij}\rceil x_j \ge \lceil\lambda_i b_i\rceil = 1$ because $\lceil \lambda^\top A \rceil \ge \lceil{\lambda_i A^i}\rceil$.
\end{proof}

\begin{proposition}
\label{prop:coverCGwithbounds}
	For a covering polyhedron with bounds and a non-negative function $c \in \R^n_+$, we have $z^{1\C} \geq \frac{1}{2} z^{\C}$.
\end{proposition}

\begin{proof}
For a covering polyhedron with bounds $P$, let $\bar P = P + \mathbb{R}^n_+$.
By applying Proposition~\ref{prop:coverUpward} recursively, $\bar P$ is a covering polyhedron with bounds.
Moreover, each covering inequality of $\bar P$ is a conic combination of one covering inequality of $P$ with the bounds $x \le u$.

We will argue bounds on the ratio between 
$$z^{\C} = \min \{c^{\top}x : x \in \C(P)\} \quad \text{and} \quad z^{1\C} = \min \{c^{\top}x : x \in 1\C(P)\}$$
by using known bounds on the ratio between covering problems 
$$\bar z^{\C} = \min \{ c^{\top}x : x \in \C(\bar P)\} \quad \text{and} \quad \bar z^{1\C} = \min \{ c^{\top}x : x \in 1\C(\bar P)\}.$$

We will show $z^{\C} \le \bar z^{\C}$ and $z^{1\C} \ge \bar z^{1\C}$.
Together with a bound of $2$ on the ratio for covering problems from Proposition \ref{prop:coverCGnobounds}, this implies the same bound on the ratio for covering problems with bounds:
\begin{align*}
\frac{z^{\C}}{z^{1\C}} \le \frac{\bar z^{\C}}{\bar z^{1\C}} \le 2.
\end{align*}


\paragraph{Claim 1} 
\label{claim_k1} $z^{1\C} \ge \bar z^{1\C}$. 
\\ \\ \emph{Proof.} 
We only need to show that 
\begin{align}
\label{claim_k_cont}
1\C(P) \subseteq  1\C(\bar P).
\end{align}
since the relation \eqref{claim_k_cont} directly implies
\begin{align*}
z^{1\C} = \min \{ c^{\top}x : x \in 1\C(P) \} \ge  \min \{ c^{\top}x : x \in  1\C(\bar P) \} = \bar z^{1\C}.
\end{align*}
By Proposition \ref{prop:coverUpward}, every constraint of $\bar{P}$ is a conic combination of a single covering constraint and one bound constraint. Therefore it follows from the definition of $1\C(P)$ that $1\C(P) \subseteq 1\C(\bar P)$. This shows \eqref{claim_k_cont}.~$\diamond$
\paragraph{Claim 2} 
\label{claim_k2}
We have $z^{\C} \le \bar z^{\C}$. 
\\ \\ \emph{Proof.} 
Since $c \ge 0$, to prove $z^{\C} \le \bar z^{\C}$ it is sufficient to show that 
\begin{align}
\label{claim_k_cont2}
\C(P) + \mathbb{R}^n_+ \supseteq  \C(\bar P).
\end{align}
In fact, relation \eqref{claim_k_cont2} directly implies
\begin{align*}
z^{\C} = \min \{ c^{\top}x : x \in \C(P) \} = \min \{ c^{\top}x : x \in \C(P)+ \mathbb{R}^n_+  \} \le  \min \{ c^{\top}x : x \in  \C(\bar P) \} = \bar z^{\C}.
\end{align*}
In order to prove relation \eqref{claim_k_cont2}, we prove that 
\begin{align}
\label{this}
\C(P) + \textup{cone}(e_j) \supseteq \C(P + \textup{cone}(e_j)).
\end{align}
In fact, by Proposition \ref{prop:coverUpward}, $P + \textup{cone}(e_j)$ is also a covering polyhedron with bounds.
Therefore we can apply relation \eqref{this} recursively (for example, for $j \neq j'$, we have $\C(P) + \textup{cone}(e_j) + \textup{cone}(e_{j'})\supseteq \C(P + \textup{cone}(e_j))+ \textup{cone}(e_{j'}) \supseteq \C(P + \textup{cone}(e_j) + \textup{cone}(e_{j'}))$), and we obtain $\C(P) + \mathbb{R}^n_+ \supseteq \C(P + \mathbb{R}^n_+) =  \C(\bar P)$, thus \eqref{claim_k_cont2}.

If $P = P + \textup{cone}(e_j)$, then \eqref{this} follows easily, therefore we now assume that $u_j$ is finite, and therefore by assumption integral. By definition,
$$\C(P) = P \cap \{x: a^{\top}x \ge \lceil{\beta}\rceil, \text{ where } a^{\top}x \ge \beta \text{ valid for } P, \ a \in \mathbb{Z}^n\}.$$

We show that all inequalities with $a_j < 0$ can be dropped from such definition. More precisely:
$$\C(P) = P \cap \{x: a^{\top}x \ge \lceil{\beta}\rceil, \text{ where } a^{\top}x \ge \beta \text{ valid for } P, \ a \in \mathbb{Z}^n, \ a_j \ge 0\}.$$
Let $a^{\top}x \ge \beta$ be valid for $P$, with  $a \in \mathbb{Z}^n$ and $a_j < 0$.
Now consider the inequality $(a')^{\top}x \ge \beta'$ obtained as the sum of $ax \ge \beta$ and $-a_j x_j \ge -a_j u_j$.
Note that $a' \in \mathbb{Z}^n$ and $a'_j = 0$.
We next verify that $(a')^{\top}x \ge \beta'$ is valid for $P$. In particular, if $\hat{x}:= (\hat{x}_j, \hat{x}_{-}) \in P$ (here the subscript $\textup{ }_{-}$ denotes all components other than $j$), then $(u_j, \hat{x}_{-}) \in P$ and therefore, $a^{\top}_{-}\hat{x}_{-} + a_ju_j \geq \beta $. Equivalently, $a^{\top}_{-}\hat{x}_{-} \geq \beta - a_j u_j$ or $(a')^{\top}\hat{x} = a^{\top}_{-}\hat{x}_{-} \geq \beta - a_j u_j = \beta'$.

Moreover, note that $(a')^{\top}x \ge \lceil{\beta'}\rceil$ cuts from $P$ at least all the points cut by $(a)^{\top}x \ge \lceil\beta \rceil$. To see this, suppose $\hat{x}:= (\hat{x}_j, \hat{x}_{-}) \in P$ is separated by $(a)^{\top}x \ge \lceil\beta \rceil$. Then
$a^{\top}_{-}\hat{x}_{-} + a_ju_j \leq a^{\top}_{-}\hat{x}_{-} + a_j x_j <\lceil\beta \rceil$, since $a_j \leq 0$ and $\hat{x}_j \in P$. Equivalently, $(a')^{\top}\hat{x} = a^{\top}_{-}\hat{x}_{-} < \lceil\beta \rceil - a_ju_j = \lceil \beta - a_ju_j \rceil = \lceil{\beta'}\rceil$, since $a_j u_j \in \mathbb{Z}$.

Therefore
\begin{align*}
\C(P) & = P \cap \{x\,|\, ax \ge \lceil{\beta}\rceil, \text{ where } ax \ge \beta \text{ valid for } P, \ a \in \mathbb{Z}^n, \ a_j \ge 0\} \\
& = P \cap \{x\,|\, ax \ge \lceil{\beta}\rceil, \text{ where } ax \ge \beta \text{ valid for } P + \textup{cone}(e_j), \ a \in \mathbb{Z}^n\} \\
& = P \cap \C(P + \textup{cone}(e_j)) \\
& = \{x \,|\, x_j \le u_j\} \cap \C(P + \textup{cone}(e_j)),
\end{align*}
where the last equation follows from the fact that if $y \in (P + \textup{conv}(x_j)) \cap  \{x \,|\, x_j \le u_j\}$, then $y \in P$. Thus we obtain 
$$\C(P) + \textup{cone}(e_j) = \left(\{x : x_j \le u_j\} \cap \C(P + \textup{cone}(e_j))\right)+ \textup{cone}(e_j).$$

Finally, we show that 
\begin{eqnarray}
\left(\{x : x_j \le u_j\} \cap \C(P + \textup{cone}(e_j))\right)+ \textup{cone}(e_j)\supseteq \C(P + \textup{cone}(e_j)),
\end{eqnarray}
to complete the proof.

First we verify that if $\hat{x} := (\hat{x}_{-}, \hat{x}_j) \in \C(P + \textup{cone}(e_j))$ and $\hat{x}_j \geq u_j$, then $(\hat{x}_{-}, u_j) \in \C(P + \textup{cone}(e_j))$. Assume by contradiction that $c_{-}^{\top}{x}_{-} + c_j {x}_j \geq \delta$ be a valid inequality for $P + \textup{cone}(e_j)$ with $c \in \mathbb{Z}^n$ such that $c_{-}^{\top}\hat{x}_{-} + c_j u_j  < \lceil\delta \rceil$. We will show that the point $(\hat{x}_{-}, \hat{x}_j)$ also does not belong to $\C(P + \textup{cone}(e_j))$ to obtain a contradiction. Note first that $c_{-}^{\top}{x}_{-} + c_j {x}_j \geq \delta$ is a valid inequality for $P$ with $c_j \geq 0$. Therefore $c_{-}^{\top}{x}_{-} \geq \delta - c_j u_j$ is a valid inequality for $P$. However since the $j^{\text{th}}$ component of $c':= (c_{-}, 0)$ is non-negative, we have that  $c_{-}^{\top}{x}_{-} \geq \delta - c_j u_j$ is a valid inequality for $P + \textup{cone}(e_j)$. In other words,  $c_{-}^{\top}{x}_{-} \geq \lceil \delta - c_j u_j \rceil = \lceil\delta\rceil - c_ju_j$ is a CG inequality for $P + \textup{cone}(e_j)$. However note that this CG inequality separates the point $(\hat{x}_{-}, \hat{x}_j)$.

Now let $\hat{x} := (\hat{x}_{-}, \hat{x}_j) \in \C(P + \textup{cone}(e_j))$. If $\hat{x}_j \leq u_j$, then clearly $\hat{x} \in (\{x : x_j \le u_j\} \cap \C(P + \textup{cone}(e_j)))+ \textup{cone}(e_j)$. In $\hat{x}_j \geq u_j$, then based on the above discussion $(\hat{x}_{-}, u_j) \in \C(P + \textup{cone}(e_j))$. In other words, $(\hat{x}_{-}, u_j) \in \left(\{x : x_j \le u_j\} \cap \C(P + \textup{cone}(e_j))\right)$. Thus 
$$\hat{x} = (\hat{x}_{-}, u_j) + (0, x_j - u_j) \in \left(\{x : x_j \le u_j\} \cap \C(P + \textup{cone}(e_j))\right)+ \textup{cone}(e_j),$$ completing the proof. 

Therefore we have proven the claim by showing \eqref{this}.
~$\diamond$
\\  \\ 
This concludes the proof of Proposition \ref{prop:coverCGwithbounds}.
\end{proof}
\paragraph{Tight examples.} We need to show that for there is an instance where $\frac{z^{\C}}{z^{1\C}} \ge 2-\e$. But the proof of Proposition \ref{prop:LBAC} already shows that this happens for the instance given by \eqref{eq:LBAC}. 

	We next show that $z^{\C}$ can be arbitrarily bad in comparison to $z^{KC}$.
	
	\begin{proposition}	
	\label{prop:coverCGvsKC}
	$z^{\C}$ can be arbitrarily bad in comparison to $z^{KC}$ for 0-1 covering problems.
	\end{proposition}
	\begin{proof}
	
	Consider the problem 
	\begin{align*}
		\min ~&x_n\\
		st ~& x_1 + \ldots x_{n-1} + n x_n \ge n\\
			&x \in \{0,1\}^n.
	\end{align*}
	It is straightforward to verify that the CG closure of this problem should be obtained by just adding the inequality $\lceil 1/n \rceil x_1 + \ldots \lceil 1/n \rceil x_{n-1} + x_n \ge 1 \equiv \sum_i x_i \ge 1$. So optimizing over the CG closure gives value $1/n$. 
	But the 1-row-aggregated closure gives the integer hull, so optimizing over it gives value~$1$.
\end{proof}


\subsection{Proof of Theorem \ref{thm:rankCover}}
\label{subsec:CoveringRankProof}
We use the following result on bounds of integrality gap of covering IPs as a function of the number of constraints.
\begin{theorem}[\cite{vazirani2013approximation}]
\label{thm:vazirani}
Consider a covering IP of the following form: $\min \{ c^\top x \mid Dx \geq f,~ x \in \Z^n_+\}$, where $D \in \R_+^{k \times n}$ such that $D_{ij} \leq f_i$ for all $i \in[k],~ j\in [n]$, and $c \in \R^n_+$. Then\footnote{The constant 8 can be easily verified using the proof techniques in \cite{vazirani2013approximation}.}, $z^I \leq 8 \log_2 (2k) z^{LP}$.
\end{theorem}

\paragraph{Lower bound on rank.} We will prove Theorem \ref{thm:rankCover} for a more general non-empty covering \emph{set} $Q = \{x \in \R^n_+ \mid \Ai x \geq b_i,~i \in \mathcal{I} \}$, where $\mathcal{I}$ is an arbitrary index set and $0 \leq A^i_j \leq b_i$ for all $i \in \mathcal{I}$. We will call a covering set with these properties a \emph{well-behaved covering set}. 

	Given a matrix $(D,f) \in \R^{k \times n} \times \R^k$, we say that it is \emph{\kvi \ for $Q$} if $(D,f)$ is non-negative and the $k$ inequalities $D^i x \ge f_i$ are valid for $Q$. 
	We denote the polyhedral outer-approximation $\{ x \in \R^n_+ \mid D x \geq f \}$ of $Q$ by $P_{(D,f)}$. Then by definition 
	\begin{align*}
		\A_k(Q) = \bigcap_{(D,f) \textrm{ is a \kvi \ for $Q$}} (P_{(D,f)})^I.
	\end{align*}

	It will be important to show that if $Q$ is well-behaved, then so is the closure $\A_k(Q)$. For that we need the following observation.

\paragraph{Claim 1} Consider a well-behaved covering set $Q$ and let $\alpha^\top x \geq \beta$ be a valid inequality for it. Then, there exists a valid inequality $\hat{\alpha}^\top x \geq \hat{\beta}$ for $Q$ with the following properties: (i) $\hat{\alpha}_j \leq \alpha_j$ for all $j \in [n]$ (ii) $\hat{\beta} \geq \beta$ (iii) $\hat{\alpha}_j \leq \hat{\beta}$ for all $j \in [n]$.
\\ \\ \emph{Proof.} 
As $Q \neq \emptyset$, by the generalized Farkas Lemma (Theorem 3.1 in \cite{lopez1998linear}), $\alpha^\top x \geq \beta$ is a valid inequality for $Q$ if and only if  
\begin{eqnarray*}
\left[
\begin{array}{c}
\alpha \\ \beta
\end{array}
\right]
 \in \text{cl} \Bigg( \text{cone} \Bigg( \Bigg\{ 
 \left[
\begin{array}{c}
\boldsymbol{0} \\ -1
\end{array}
\right]\Bigg\} \ \cup \ 
\Bigg\{
\left[
\begin{array}{c}
(\Ai)^\top \\ b_i
\end{array}
\right]
;~i \in \mathcal{I}
  \Bigg\} \ \cup \ \Bigg\{
 \left[\begin{array}{c}
e_j \\ 0
\end{array}\right];~j \in [n]  
   \Bigg\} \Bigg) \Bigg)
 := F,
\end{eqnarray*}
where $\text{cl}$ and $\boldsymbol{0}$ stands for the closure and the vector of zeros in $\R^n$, respectively. 
We also let 
\begin{eqnarray*}
G :=
\text{cl} \Bigg( \text{cone} \Bigg( \Bigg\{ 
\left[
\begin{array}{c}
(\Ai)^\top \\ b_i
\end{array}
\right]
;~i \in \mathcal{I}
  \Bigg\} \Bigg) \Bigg), ~~~H := \cone \Bigg( 
 \left[
\begin{array}{c}
\boldsymbol{0} \\ -1
\end{array}
\right] \ \cup \ 
\Bigg\{
 \left[\begin{array}{c}
e_j \\ 0
\end{array}\right];~j \in [n]  
   \Bigg\} \Bigg)
\end{eqnarray*}
Note that $F = \text{cl}(G + H)$. We will show that $G+H$ is closed, thereby implying that $F = G+H$. 

For that, notice that the cones $G$ and $H$ are \emph{positively semi-independent}, that is if $g \in G$ and $h \in H$ satisfy $g + h = [\boldsymbol{0}, 0]$, then $g=h= [\boldsymbol{0},0]$: To see this, consider a vector $[a, -b] \in H$, so $a$ and $b$ are non-negative, such that $[-a,b]$ belongs to $G$. Since $\Ai$ is non-negative for all $i \in \mathcal{I}$, this implies that $a= 0$. Furthermore, as $Q$ is non-empty and the inequality $-a^\top x \ge b$ is valid for $Q$, we obtain $b=0$, which concludes the argument. 


 Hence, by a result in \cite{gale1976malinvaud}, $G+H$ is closed, and so $F = G + H$.
This implies that if $\alpha^\top x \geq \beta$  is a valid inequality for $Q$, then 
\begin{eqnarray*}
\left[
\begin{array}{c}
\alpha \\ \beta
\end{array}
\right]
= 
\left[
\begin{array}{c}
\hat{\alpha} \\ \hat{\beta}
\end{array}
\right]
+
\lambda
\left[
\begin{array}{c}
\boldsymbol{0} \\ -1
\end{array}
\right]
+
\mu_j
\left[
\begin{array}{c}
e_j \\ 0
\end{array}
\right],
\quad \lambda \geq 0,
~ \mu_j \geq 0 \ \forall j \in [n]
,
\end{eqnarray*}
where $[\hat{\alpha}, \hat{\beta}] \in G$. Note that $\hat{\alpha}^\top x \geq \hat{\beta}$ is a valid inequality for $Q$. Moreover, $[\hat{\alpha}, \hat{\beta}] \in G$ implies that $\hat{\alpha}_j \leq \hat{\beta}$ for all $j \in [n]$. Also, all the other conditions of the claim are satisfied which completes the proof.
~$\diamond$
\paragraph{Claim 2} 
If $Q$ is a well-behaved covering set, then $\A_k(Q)$ is also a well-behaved covering set.
\\ \\ \emph{Proof.} 
Given $(D,f)$ a \kvi \ for $Q$, let $(\hat{D},\hat{f})$ be obtained as in Claim 1. Then, by the previous claim, $(\hat{D}, \hat{f})$ is a \kvi \ for $Q$. Observe that by construction of $(\hat{D},\hat{f})$ we have $P_{(D,f)} \supseteq P_{(\hat{D},\hat{f})}$, and therefore $(P_{(D,f)})^I \supseteq (P_{(\hat{D},\hat{f})})^I$. Hence
\begin{equation}
\label{eq:Df_forPI}
\A_k(Q) = \bigcap_{(D,f) \ \text{is a \kvi \ for $T$}} (P_{(\hat{D},\hat{f})})^I.
\end{equation}

To show that $\A_k(Q)$ is well-behaved, it suffices to show that $(P_{(\hat{D},\hat{f})})^I$ is of the form $\{ x \in \R^n_+ \mid R^i x \geq s_i,~ i \in [m']  \}$, where $0 \leq R_{ij} \leq s_i$ for all $j \in [n],~ i \in [m']$.
Since the recession cone of $P_{(\hat{D},\hat{f})}$ is $\R^n_+$, and $P_{(\hat{D},\hat{f})}$ is a polyhedron, by Theorem 6 in \cite{convex2013deymoran}, $(P_{(\hat{D},\hat{f})})^I$ is a rational polyhedron with the same recession cone as $P_{(\hat{D},\hat{f})}$. Hence, $R$ is non-negative. Moreover, we may take the inequalities $R^i x \geq s_i,~ i \in [m']$ to be the facet-defining inequalities that satisfy at $n+1$ affinely independent integer points at equality. To show $R_{i,j^*} \leq s_i$ for some $j^* \in [n]$, observe that in particular there exists an integer point $\hat{x}$ among these $n+1$ affinely independent ones satisfying $\hat{x}_{j^*} \geq 1$ and $\sum_{j \in [n]} R_{ij} \hat{x}_j = s_i$ (else all these points would satisfy the additional equation $x_j = 0$ and live in an $(n-2)$-dimensional space, contradicting their affine independence). This implies $R_{i,j^*} = \frac{s_i-\sum_{j \neq j^*} R_{ij} \hat{x}_j}{\hat{x}_{j^*}} \leq s_i$.
~$\diamond$
\\  \\ 
Let $Q^\ell$ be the $\ell^{\text{th}}$ $k$-aggregation closure of $Q$.
\paragraph{Claim 3} 
$Q^\ell \subseteq \frac{1}{8 \log_2(2k)} Q^{\ell+1}$. 
\\ \\ \emph{Proof.} 
By the previous claim, $Q^\ell$ is a well-behaved covering set. For every $(D,f)$ \kvi \ for $Q^\ell$, by Theorem \ref{thm:vazirani} we have 
$$Q^\ell \subseteq P_{(\hat{D},\hat{f})} \subseteq \frac{1}{8 \log_2(2k)} (P_{(\hat{D}, \hat{f})})^I.$$
 By Observation~\ref{obs:bijection} we have that
\begin{align*}
\begin{split}
Q^\ell \ \subseteq & \ \frac{1}{8 \log_2(2k)} \bigcap_{(D,f) \textup{ is \kvi \ for } Q^\ell} (P_{(\hat{D}, \hat{f})})^I \\
= & \frac{1}{8 \log_2(2k)} Q^{\ell+1},
\end{split}
\end{align*}
where the last equality follows from \eqref{eq:Df_forPI}.
~$\diamond$
\\  \\ 
\indent
Using an argument similar to the proof of Theorem \ref{thm:packrankgen} (employing now Proposition \ref{prop:blowupCover}), we obtain that the rank of the $k$-aggregation closure is at least $\left \lceil  \left(\frac{\textup{log}_2\left(\frac{z^I}{z^{LP}}\right)}{3+\textup{log}_2\textup{log}_2(2k)}\right)\right \rceil$. 
\\ \\ \\ \\
\textbf{Acknowledgements.} 
Santanu S. Dey would like to acknowledge the support of the NSF grant CMMI\#1149400.

\newpage
\begin{appendix}
\section{Polyhedrality of aggregation closure for dense IPs}
\label{sec:appendix}
We prove the result for the case of covering IPs and a similar proof can be given for the packing case. 

\begin{proposition}
Let $Q = \{ x \in \R^n_+ \mid Ax \geq b \}$ be a covering polyhedron with $A \in \Z_+^{m \times n},~b \in \Z_+^{n}$, $A_{ij} \geq 1$ for all $i \in [m],~j\in [n]$, and $b_i \geq 1$ for all $i \in [m]$. Then, $\A_k(Q)$ is a polyhedron.
\end{proposition}

\begin{proof}
The intercept of the hyperplane corresponding to the $i^{\text{th}}$ constraint, $\Ai x \geq b_i$, of the $j^{\text{th}}$ coordinate axis is $\frac{b_i}{A_{ij}}$. It is straightforward to verify that the intercept of any aggregated constraint on the $j^{\text{th}}$ coordinate axis belongs to the set $\left[\min_{i \in [m]} \frac{b_i}{A_{ij}},\max_{i \in [m]} \frac{b_i}{A_{ij}}\right]$. Let $M = \max_{i \in [m],~j \in [n]} \frac{b_i}{A_{ij}}$ and let $T = [0,M]^n \cap \Z_+^n$.

Based on the above observation, the set of integer points contained in $\{ x \in \R^n_+ \mid (\lambda^\ell)^\top A x \geq (\lambda^\ell)^\top b,~ \ell \in [k] \}$ is of the form $S \cup (\Z_+^n \setminus T)$ where $S \subseteq T$. Since $T$ is a finite set, this completes the proof as the number of distinct integer hulls obtained from $k$-aggregations is finite.
\end{proof}


\section{Proof of Proposition \ref{prop:approxPack}} \label{app:charApproxPacking}

	Given a convex set $C \subseteq \R^n$, its support function $\delta^*(. \mid C)$ is defined by $\delta^*(c \mid C) = \sup \{ c^T x \mid x \in C\}$.
	
	Consider packing sets $U \supseteq V$. Since $U$ and $V$ are closed, from Corollary 13.1.1 of \cite{rockafellar:1970} we have that $U \supseteq \alpha V$ iff 
	\begin{align}	
		\sup_{c \in \R^n} \left(\delta^*(c \mid U) - \delta^*(c \mid \alpha V)\right) \le 0. \label{eq:rock}
	\end{align} 
	Since $U$ is a packing set we have the following property. Consider a vector $c \in \R^n$, let $I$ be the index of its negative components, and let $\tilde{c}$ be obtained by changing the components of $c$ in $I$ to $0$. Then $\delta^*(c \mid U) = \delta^*(\tilde{c} \mid U)$: the direction ``$\le$'' follows from $U \subseteq \R^n_+$; the direction ``$\ge$'' holds because for every point $x \in U$, if we construct $\tilde{x}$ by changing the components in $I$ of $x$ to $0$ then $\tilde{x} \in U$ and $c^T \tilde{x} = \tilde{c}^T x$. Since the same holds for $\alpha V$, we have that in equation \eqref{eq:rock} we can take the supremum over only non-negative $c$'s, and hence it holds iff for all $c \in \R^n_+$, $\delta^*(c \mid U) \le \delta^*(c \mid \alpha V)$. But since $\delta^*(c \mid \alpha V) = \alpha\,\delta^*(c \mid V)$ (Corollary 16.1.1 of \cite{rockafellar:1970}), this happens iff for all $c \in \R^n_+$, $\delta^*(c \mid U) \le \alpha\,\delta^*(c \mid V)$. This concludes the proof.

\section{Proof of Proposition \ref{prop:inthullPack}} \label{app:packinginthull}

Let $Q = \{ x \in \R_+^n \mid \Ai x \leq b_i \ \forall i \in I \}$. We assume that for all $j \in [n]$, there exists $i \in I$ with $A_{ij} > 0$. Otherwise, we can project out the $j^{\text{th}}$ variable and continue with the argument as the $j^{\text{th}}$ variable is allowed to take any value. Therefore, $Q$ is a bounded set and $Q^I$ is a polyhedron. Let $Q^I = \{ x \in \R_+^n  \mid Cx \leq d \}$. We next argue that $C$ and $d$ are non-negative to complete the proof.   

Note that since $\boldsymbol{0} \in Q$, $d \geq 0$. The fact that we can take $C \geq 0$ follows from the following claim.

\paragraph{Claim.}
Let $C^i x \leq d_i$ be a facet-defining inequality for $Q^I$ and $C_{i j^{*}} < 0$ for some $i$ and $j^{*}$. Define a vector $\hat{c}$ as $\hat{c}_{j^{*}} = 0$ and $\hat{c}_{j} = C_{ij}$ for all other $j$. Then  
$\hat{c} x \leq d_i$ is valid for $Q^I$.  
\\ \\ \emph{Proof.} 
	  Assume by contradiction that there exists $\hat{x} \in Q \cap \Z^n $ such that $\sum_{j = 1}^n \hat{c}_{j}\hat{x}_j > d_i$. Since $Q$ is a packing set, we have that $\tilde{x} \in Q \cap \Z^n$, where $\tilde{x}$ is defined as $\tilde{x}_j = \hat{x}_j$ for all $j \in [n] \setminus \{j^{*}\}$ and $\tilde{x}_j^{*} = 0$. Then $d_i < \sum_{j = 1}^n \hat{c}_{j}\hat{x}_j = \sum_{j = 1}^n \hat{c}_{j}\tilde{x}_j  = \sum_{j = 1}^n {C}_{ij}\tilde{x}_j \leq d_i$, a contradiction. ~$\diamond$
\\ 
\end{appendix}
\bibliographystyle{plain}
\bibliography{test}   

\end{document}